\documentclass{amsart}[12pt]
\usepackage{mystyle}

\newcommand{\CP}{\mathbb{C}\mathbb{P}}

\newcommand{\grid}{\mathrm{grid}}
\newcommand{\sing}{\mathrm{Sing}}
\newcommand{\smooth}{\mathrm{Smooth}}
\newcommand{\err}{\mathrm{err}}
\newcommand{\BSG}{\mathrm{BSG}}
\DeclareMathOperator{\Div}{Div}

\title{Sylvester-Gallai configurations on algebraic curves in $\C^2$}
\author{Alex Cohen}
\date{August 2025}

\begin{document}
\begin{abstract}
The Sylvester-Gallai theorem says that for any finite set of non-collinear points in $\R^2$, there is some line passing through exactly two points of the set. Over the complex numbers, this theorem fails: there are finite configurations with the property that any line through two points also passes through a third. Only one infinite class of examples (the Fermat configurations) is known, and it is a folklore conjecture that this is the only infinite class of examples. We prove this conjecture in the ``99\% structure'' case where we assume most of the points lie on a low degree algebraic curve. 
\end{abstract}

\maketitle 

\section{Introduction}

The Sylvester-Gallai theorem says that among any finite set of non-collinear points in $\R^2$, there exists a line passing through exactly two of them. These are called \textit{ordinary lines}. A \textit{Sylvester--Gallai} configuration in $\C^2$ is a finite set of points that is not collinear, but such that every line determined by two of its points passes through a third. That is, the set contains no ordinary lines. Only a few such configurations are known:
\begin{enumerate}[label=(\roman*)]
	\item The \textit{Fermat} configuration on $3n$ points lies on three non concurrent lines. It consists of the $3n$ inflection points of the Fermat curve $x^n + y^n = z^n$. In coordinates, 
	\begin{align*}
		A &= \{a_1, \ldots, a_n\} \cup \{b_1, \ldots, b_n\} \cup \{c_1, \ldots, c_n\} \\ 
		a_j &= [0 : -\zeta^j : 1],\quad b_j = [-\zeta^j : 0 : 1],\quad c_j = [1 : -\zeta^j : 0],
	\end{align*}
	where $\zeta$ is a primitive $n$th root of unity. The points $a_r, b_s, c_t$ are collinear if and only if $r+s=t \pmod n$. \label{item:fermat_config_intro}

	\item There is an exceptional configuration due to Klein on 21 points. 

	\item There is an exceptional configuration due to Wiman on 45 points. 
\end{enumerate}
The Klein and Wiman configurations come from reflection groups of $\C^3$, see~\cite{Hirzebruch1983}*{Section 1.2} and~\cite{PokoraEtAl}*{Section 4}.
Based on these examples, we may hope for:
\begin{conjecture}
Every large enough Sylvester--Gallai configuration in $\CP^2$ is projectively equivalent to a Fermat configuration.
\end{conjecture}
This is a folklore conjecture, but I have not found it stated in the literature.
If we impose the constraint that most of $A$ lies on a low degree curve, we can prove this conjecture.
\begin{theorem}\label{thm:MainThm}
For every $d \geq 2$, there exists $\varepsilon(d) > 0$ and $n_0(d) > 0$ such that the following holds. 
Let $A \subset \CP^2$ be a Sylvester-Gallai configuration of $n \geq n_0$ many points, all but $\varepsilon n$ of which lie on an algebraic curve of degree $\leq d$. Then $A$ is projectively equivalent to a Fermat configuration.
\end{theorem}
Our algebraic curves may be reducible. For example, the Fermat configurations lie on a union of three lines, which is a cubic curve. 

Cubic curves come up naturally in Sylvester--Gallai theory, because collinear triples in a cubic curve have abelian group structure. See \Cref{thm:GpStructureCubic} for a precise formulation.
Our starting point is the following theorem of Raz, Scharir, and de Zeeuw, which locates cubic curve structure in $A$.  
\begin{theorem}[\cite{RazSharirDeZeeuw2016}*{Theorem 6.1}] \label{thm:raz_sharir_dezeeuw}
Suppose $C_1, C_2, C_3$ are (not necessarily distinct) irreducible algebraic curves of degree $\leq d$ and $S_1 \subset C_1, S_2 \subset C_2, S_3 \subset C_3$ are finite sets. Then the number of proper collinear triples in $S_1 \times S_2 \times S_3$ is 
\begin{equation*}
    O_d(|S_1|^{1/2}|S_2|^{2/3}|S_3|^{2/3} + |S_1|^{1/2}(|S_1|^{1/2}+|S_2|+|S_3|))
\end{equation*}
unless $C_1 \cup C_2 \cup C_3$ is a line or a cubic curve. 
\end{theorem}
Using \Cref{thm:raz_sharir_dezeeuw}, we can show that most of $A$ lies on either a union of lines, or a union of lines and one conic, or a single irreducible cubic curve. See \Cref{lem:CoarseStructureLemma} for the precise statement. 

The rest of this paper is based on Green and Tao's analysis in~\cite{GreenTao2013}. In~\cite{GreenTao2013}, Green and Tao show that if a subset of $\R^2$ has $O(n)$ ordinary lines, it must come from group structure on a cubic curve---see~\cite{GreenTao2013}*{Theorem 1.5}. To prove this theorem, Green and Tao first use topological arguments to prove that most of their set lies on a union of not-too-many cubic curves, and then use additive combinatorics to analyze how collinear triples between the various lines, conics, and irreducible cubics in the set may interact. 

For example, suppose $A$ is a set of $n$ points lying on a union of $4$ generic lines, $\ell_1,\ell_2,\ell_3,\ell_4$, each of which contain around a fourth of the points of $A$. 
Suppose that between any three of these lines, there are $\geq \kappa n^2$ collinear triples. To understand the collinear triples between the first three lines, we apply a projective transformation so that 
\begin{align}\label{eq:3linesExampel}
\ell_1 = \{[0 : x : 1]\},\quad \ell_2 = \{[1 : 0 : y]\},\quad \ell_3 = \{[-z : 1 : 0]\}.
\end{align}
We chose these coordinates so that 
\begin{align}\label{eq:Collinear3Nonconcurrent}
    [0 : x : 1], [1 : 0 : y],\text{ and } [-z : 1 : 0]\text{ are collinear if and only if $xyz = 1$.}
\end{align}
If we identify $A\cap \ell_j$ with $\widetilde A_j \subset \C^{\times} := \C \setminus \{0\}$ via these coordinates, then 
\[
\#\{(x,y,z) \in \widetilde A_1 \times \widetilde A_2\times \widetilde A_3\, :\, xyz = 1\} \geq \kappa|A|^2.
\]
By the Balog-Szemer\'edi-Gowers Theorem (\Cref{thm:BSG}), there are subsets $\widetilde{A_j}' \subset \widetilde{A_j}$, which have size $\widetilde A_j' \geq \kappa^{-C_{\BSG}} |\widetilde A_j| $ and do not expand very much upon taking a product with itself:
\[
\widetilde A_j' \cdot \widetilde A_j' \leq \kappa^{-C_{\BSG}} |\widetilde A_j'|.
\]
The corresponding set $A_j' \subset \ell_1$ is friendly with $\ell_2$ and $\ell_3$. But $A_j'$ also has to interact with $\ell_4$. There have to be many collinear triples between 
\[
A_1'\times (A\cap \ell_2)\times (A\cap \ell_4)\quad\text{or}\quad A_1'\times (A\cap \ell_3)\times (A\cap \ell_4).
\]
Assume we are in the first case, and apply a projective transformation mapping $\ell_1,\ell_2,\ell_4$ to the three lines above. This has the effect of changing coordinates on $\ell_1$ by a M\"obius transformation $\psi$ that does not fix $\{0,\infty\}$. Applying the Balog--Szmer\'edi--Gowers theorem again, we find a large subset $\widetilde{A_1}'' \subset \widetilde{A_1}$ such that 
\[
|\widetilde{A_1}''\cdot \widetilde{A_1}''| \leq \kappa^{-10C_{\BSG}^2} |\widetilde{A_1}''|\quad\text{and}\quad |\psi(\widetilde{A_1}'')\cdot \psi(\widetilde{A_1}'')| \leq \kappa^{-10C_{\BSG}^2} |\widetilde{A_1}''|.
\]
We prove a variant of the sum-product theorem in \Cref{lem:expanding_lem} that rules this out. 
Using this sort of analysis in a parallel way to Green and Tao, we can show that most of $A$ lies either on a single cubic curve, or on a family of concurrent lines. 

At this point there is a wrinkle. Green and Tao deal with the concurrent line case using convexity (their paper uses an argument of Luke Alexander Betts suggested on Tao's blog). That argument breaks down over the complex numbers. Luckily, I proved a few years ago that there are no Sylvester--Gallai configurations lying on a family of concurrent lines. This result also ultimately uses convexity, but in a different way---the key estimate uses the divergence theorem applied to a convex, real-valued function on $\C$. 
\begin{theorem}[\cite{Cohen2022}*{Theorem 2}]\label{thm:SG_concurrent_lines}
If a non-collinear set $A \subset \C^2$ lies on a family of $m$ concurrent lines, and one of those lines contains more than $m-2$ points (not including the point of concurrency), then $A$ admits an ordinary line.
\end{theorem}

After using \Cref{thm:SG_concurrent_lines} to deal with the concurrent line case, we are left with the case that most of $A$ lies on a cubic curve. Following Green and Tao, we use more additive combinatorics to show $A$ aligns with subgroup structure on the cubic curve. At this point we have pretty much pinned down what $A$ looks like. 
If the cubic curve consists of $3$ non-concurrent lines, we show $A$ is projectively equivalent to a Fermat configuration; otherwise, $A$ admits an ordinary line.

\section{Notation}
We work in projective space 
\[
\CP^2 = \{[x : y : z]\, :\, (x,y,z) \neq 0\}, [x : y : z] \sim [\lambda x : \lambda y : \lambda z]. 
\]
We sometimes consider $\C^2$ as a subset of $\CP^2$ with the embedding 
\[
(x,y) \mapsto [x : y : 1]. 
\]
The cardinality of a finite set $A$ is denoted $|A|$. 
\[
A \Delta B := A \cup B \setminus (A\cap B). 
\]
We denote the line through $p$ and $q$ by $\ell_{p,q}$.
\section{Abelian group structure of cubic curves}
\subsection{Geometry of plane cubics and uniformization}\label{subsec:GeomUniformization}
An algebraic curve is the vanishing locus of a homogenous polynomial,
\[
V = \{[x : y : z]\, :\, P(x,y,z) = 0\}.
\]
An irreducible algebraic curve of degree $d$ is the vanishing locus of a homogenous, degree $d$ irreducible polynomial. Any algebraic curve is a finite union of irreducible curves, and the degree is the sum of the degrees of the components. 

The \textit{singular locus} of a possibly reducible curve, denoted $\sing(V)$, is the set of points where $\nabla P$ vanishes. This is the union of the singular loci of each component and the intersection points of different components. The smooth locus $\smooth(V)$ is the complement. 

Irreducible curves of degree at most $3$ fall into five categories, and their smooth parts admit a complex analytic uniformization. We start with the smooth curves.  
\begin{itemize}
    \item A degree one curve is a \textit{line}. Every line is biholomorphic to the extended complex plane $\widehat{\C} \cong \CP^1$. 

    \item An irreducible degree two curve is a \textit{conic}. Every conic is smooth and biholomorphic to $\widehat{\C}$. 
    
    After a projective change of variables, conics may be put in the normal form
    \[
    C = \{[x : y : z]\, :\, yz = x^2\}.
    \]
    In this form, the maps
    \[
    [x : y : z] \mapsto [y : z],\qquad [t : u] \mapsto [tu : t^2 : u^2]
    \]
    exhibit a biholomorphism with $\widehat{\C}$. 

    \item A smooth irreducible cubic curve \( E \subset \CP^2 \) is called an \textit{elliptic curve}. As a Riemann surface, \( E \) is biholomorphic to a complex torus \( \C / \Lambda \), where \( \Lambda \subset \C \) is a rank-two lattice. The biholomorphism is given by the classical Abel--Jacobi map.
\end{itemize}
In affine coordinates, suppose a cubic curve has a singular point at the origin. Let \( f(X, Y) \) be its defining polynomial with \( f(0,0) = 0 \) and \( \nabla f(0,0) = 0 \). Then \( f \) has no degree zero or one terms, and we may write
\[
f(X, Y) = f_2(X, Y) + f_3(X, Y),
\]
where \( f_2 \) and \( f_3 \) are homogeneous of degrees 2 and 3 respectively.
If \( f_2 \) factors into two distinct linear forms the curve is nodal, and if \( f_2 \) is a repeated linear form the curve is cuspidal.
\begin{itemize}
\item The smooth part of a nodal cubic is biholomorphic to \( \C^\times := \C \setminus \{0\} \). This can be seen by projecting from the singular point: a generic line through the node intersects the curve at the node (with multiplicity 2) and at one additional smooth point. Two exceptional lines intersect the curve at the node with multiplicity 3 and do not meet any other point. Thus, stereographic projection maps the smooth part of the curve biholomorphically to \( \CP^1 \setminus \{0, \infty\} \cong \C^\times \).

\item Similarly, the smooth part of a cuspidal cubic is biholomorphic to \( \C \): stereographic projection omits only one point, yielding \( \CP^1 \setminus \{\infty\} \cong \C \).
\end{itemize}


A connected component of the smooth part of $X$ is either the smooth part of an irreducible cubic, or a conic or line minus one or two points. In any of these cases the uniformization is also a group---one of $\C/\Lambda$, $\C$, or $\C^{\times}$. A biholomorphism of any of these groups is also an affine map of groups, so each of these curves carry a canonical affine group structure. We call a Riemann surface isomorphic to $\C / \Lambda$, $\C$, or $\C^{\times}$ an abelian curve, and summarize the different cases below. 
\begin{center}
\begin{tabular}{|l|l|}
\hline
\textbf{Abelian curve} & \textbf{Uniformization } \\
\hline
Smooth irreducible cubic & \( \C / \Lambda \) \\
Smooth part of cuspidal cubic  & \( \C \) \\
Smooth part of nodal cubic  & \( \C^{\times} \) \\
Conic or line minus one point & \( \C \) \\
Conic or line minus two points & \( \C^{\times} \) \\
\hline
\end{tabular}
\end{center}

\subsection{Abelian group structure of plane cubics}\

Let $X$ be a plane cubic. The uniformization of the smooth part of $X$ is described by the following table.

\begin{center}
\begin{tabular}{|l|l|}
\hline
\textbf{Curve Type} & \textbf{Uniformization of smooth locus} \\
\hline
Smooth irreducible cubic & \( \C / \Lambda \) \\
Nodal cubic  & \( \C^{\times} \) \\
Cuspidal cubic  & \( \C \) \\
Conic and a line intersecting in one point & \( \C \sqcup \C \) \\
Conic and a line intersecting in two points & \( \C^{\times} \sqcup \C^{\times}  \) \\
Three concurrent lines & \( \C \sqcup \C \sqcup \C \) \\
Three non-concurrent lines & \( \C^{\times} \sqcup \C^{\times} \sqcup \C^{\times} \) \\
\hline
\end{tabular}
\end{center}
In all of these cases, the uniformization is some union of copies of a fixed group $G$. The next theorem describes how this abelian group structure is related to collinearity. 

\begin{theorem}\label{thm:GpStructureCubic}
Let $X$ be a plane cubic, and let $G$ be the group appearing in the uniformization of the smooth part of $X$. There is an analytic map
\[
\rho: X \setminus \sing(X) \to G
\]
which is a biholomorphism on each component such that the following holds. 
Let $S = \{p_1,p_2,p_3\}$ be a set of $3$ distinct points on $X$ with $|S\cap V| = \deg(V)$ for each irreducible component $V$ of $X$. Then
\[
\{p_1,p_2,p_3\}\text{ are collinear if and only if } \rho(p_1) + \rho(p_2) + \rho(p_3) = 0\qquad \text{in $G$.}
\]
\end{theorem}
This result is well known, but I do not know a suitable reference, so I include a proof in \Cref{sec:ProofOfGpLaw}.

\section{Coarse structure lemma}

\begin{lemma}\label{lem:CoarseStructureLemma}
For every $d \geq 2$ and $\delta > 0$, there exists $\varepsilon(d, \delta), n_0(d, \delta) > 0$ such that the following holds. Let $A$ be a Sylvester--Gallai configuration of $n \geq n_0$ points, all but $\varepsilon n$ of which lie on the smooth locus of a degree $\leq d$ algebraic curve. Then all but $\delta n$ points of $A$ lie on the smooth locus of either 
\begin{itemize}
    \item The union of $m\leq d$ lines, or 
    \item The union of $m \leq d-1$ lines and one conic, or
    \item An irreducible cubic.
\end{itemize}
\end{lemma}
\begin{proof}
We fix $\varepsilon, n_0 > 0$ such that 
\begin{align}
    \varepsilon &< \frac{\delta^2}{2 \cdot 100^2 d^3} \quad\text{and}\quad \varepsilon < \frac{\delta}{2} , \label{eq:EpsBd}\\ 
    n_0 &> \bigl(2C_d100^2 d^3 \delta^{-2}\bigr)^{6} \qquad \text{where $C_d$ is the constant from \Cref{thm:raz_sharir_dezeeuw}.}\label{eq:N0Bd} 
\end{align}
Let 
\[
T := \{(x,y,z)\in A\times A\times A\, :\, \text{$x,y,z$ are distinct and collinear}\} 
\]
be the proper collinear triples in $A \times A\times A$. We have $|T| \geq |A|(|A|-1)$, as any pair of distinct points in $A$ determine a line through a third point of $A$. 

Let $V$ be a degree $\leq d$ algebraic curve with $|A\setminus \smooth(V)| \leq \varepsilon n$, and let 
\[
V = V_1 \cup \dots \cup V_r
\]
be the decomposition into irreducibles. Set 
\[
A_j = A\cap V_j \cap \smooth(V),\qquad A_{\err} = A \setminus \smooth(V).
\]

We say a curve $V_j$ is \textit{significant} if $|A_j| \geq \frac{\delta}{100d} n$. Order the curves so that the significant ones are $V_1,\ldots,V_s$ where $s \leq r$. The total number of points off the significant curves is $\leq \frac{\delta}{100}n + \varepsilon n < \delta n$. 

Fix $i, j \in \{1,\ldots,s\}$, and consider the collinear triples going between $A_i$ and $A_j$. Not so many of these can meet a point $p \in A_{\err}$. If $V_i$ is a line and $p$ lies on $V_i$, then there are no collinear triples, because $A_j$ is disjoint from $V_i$. Otherwise, every line joining $p$ to a point of $A_j$ passes through $\leq d$ points of $A_i$, so the number of collinear triples between $A_i, A_j, \{p\}$ is $\leq d |A_j|$. Summing over $A_{\err}$, 
\[
|T\cap (A_i\times A_j\times A_{\err})| \leq d |A_j| \varepsilon n \leq \varepsilon d  n^2. 
\]
On the other hand, every pair in $A_i\times A_j$ is incident to some collinear triple, so $|T\cap (A_i\times A_j\times A)| \geq |A_i|\,|A_j|$. As $V_i$ and $V_j$ are significant, 
\[
|A_i|\, |A_j| \geq \bigl(\frac{\delta}{100d}\bigr)^2 n^2 \geq 2\varepsilon d n^2 \qquad \text{by our choice of $\varepsilon$ in \eqref{eq:EpsBd}.}
\]
Thus by the pigeonhole principle there is some $k$ such that 
\[
|T\cap (A_i\times A_j\times A_k)| \geq \frac{1}{2d} |A_i|\,|A_j| \geq \frac{1}{2d} \bigl(\frac{\delta}{100d}\bigr)^2 n^2.
\]
By \Cref{thm:raz_sharir_dezeeuw}, if $V_i \cup V_j$ has degree $> 3$, then 
\[
|T\cap (A_i\times A_j\times A_k)| \leq C_d n^{11/6}
\]
which contradicts our choice of $n_0$ in \eqref{eq:N0Bd}. It follows that
\begin{itemize}
    \item All the significant curves have degree $\leq 3$,
    \item If one of the significant curves is an irreducible cubic, that is the only significant curve, 
    \item If one of the significant curves is a conic, all the other ones are lines,
\end{itemize}
and the result follows.
\end{proof}

\section{Intermediate structure lemma}
In this section we will use additive combinatorics and \Cref{thm:SG_concurrent_lines} to show most of $A$ lies on a single cubic curve. 

\subsection{Grids}

Given an abelian group $G$ and a finite subset $Q$, define the difference set
\[
Q - Q = \{a - b\, :\, a, b \in G\}.
\]
We say $Q \subset G$ is a \textit{$C$-grid} if 
\[
|Q-Q| \leq C|Q|. 
\]
Let $Y \subset \CP^2$ be an abelian curve (see \Cref{subsec:GeomUniformization}) and let $\rho: Y \to G$ be a uniformization, where $G$ is one of $\C / \Lambda, \C, \C^{\times}$. We say $A \subset Y$ is a \textit{$(C, Y)$-grid} if  $\rho(A) \subset G$ is a $C$-grid. This is well defined, because any biholomorphism of $G$ is a group homomorphism and thus does not change the size of the difference set. 

Observe that 
\begin{align}\label{eq:SubsetIsGrid}
   \text{If $Q$ is a grid and $Q'\subset Q$, then $Q'$ is a $C\frac{|Q|}{|Q'|}$-grid}.
\end{align}

\subsection{Additive Combinatorics}

We will need the following two ingredients from additive combinatorics. The first is the Balog--Szemer\'edi--Gowers lemma, which relates lots of collinear triples to sets of small doubling. The second is an ingredient from Green and Tao's paper on ordinary lines which relates very few non-collinear triples to subgroup structure. 

\begin{theorem}[Balog--Szemer\'edi--Gowers]\label{thm:BSG}
There exists a constant $C_{\BSG} > 0$ such that the following holds.
Let $G$ be an abelian group, and let $X, Y, Z \subset G$ be sets of size $\leq n$ such that there are $\geq \kappa n^2$ triples $(x,y,z) \in X \times Y \times Z$ such that $x+y+z = 0$. Then there is some $X' \subset X$ with $|X'| \geq \kappa^{C_{\BSG}} n$ such that $|X' - X'| \leq \kappa^{-C_{\BSG}} |X'|$, and similar for $Y$ and $Z$. 
\end{theorem}
\begin{proof}
Let 
\[
\Gamma = \{(x,y) \in X\times Y\, :\, x+y \in -Z\}. 
\]
By the hypothesis, $|\Gamma| \geq \kappa n^2$. The restricted sumset 
\[
X+_{\Gamma} Y = \{x+y\, :\, (x,y) \in \Gamma\}
\]
is contained in $-Z$. The hypotheses of~\cite{TaoVu}*{Theorem 2.29} are satisfied with $K = \kappa^{-1}$ and $K' = \frac{|Z|}{|A|^{1/2}|B|^{1/2}}$. The total number of additive triples is at most $|Z||A|$ and also at most $|Z||B|$, so $|A|, |B| \geq \kappa n^2 / |Z|$, implying $K'\leq \kappa^{-1}$ as well, as needed.
\end{proof}

\begin{proposition}[Green--Tao \cite{GreenTao2013}*{Proposition A.5}]\label{prop:close_to_subgroup_prop}
Suppose that $A,B,C$ are three subsets of some abelian group $G$, all of cardinality in the range $[n - K, n+K]$, where $K \leq \varepsilon n$ for some absolute constant $\varepsilon > 0$. Suppose that there are at most $Kn$ pairs $(a, b) \in A \times B$ for which $a+b \notin C$. Then there is a subgroup $H \leq G$ and cosets $x+H, y+H$ such that 
\begin{equation*}
	|A \Delta (x+H)|,\quad |B\Delta(y+H)|,\quad |C \Delta(x+y+H)| \leq 7K. 
\end{equation*}
\end{proposition}

\subsection{Intersection of grids}


In this section, we show that two grids lying on different abelian curves have a small intersection. The key ingredient is the following incidence estimate for algebraic curves. 

\begin{theorem}[Sheffer, Szabo, \& Zahl  \cite{ShefferSzaboZahl}]\label{thm:IncidenceEst}
For each $k\geq 1, D\geq 1$, $s \geq 1$, and $\varepsilon > 0$ there is a constant $C = C(\varepsilon, D, s, k)$ so that the following holds. Let $\mc P \subset \C^2$ be a finite set of points and let $\mc C$ be a finite set of complex algebraic curves of degree at most $D$. Suppose that $(\mc P, \mc C)$ satisfies:
\begin{itemize}
    \item For any subset $\mc P' \subset \mc P$ of size $k$, there are at most $s$ curves from $\mc C$ that contain $\mc P'$, and
    \item Any pair of curves from $\mc C$ intersect in at most $s$ points from $\mc P$. 
\end{itemize}
Then 
\[
\#\{(p, C) \in \mc P\times \mc C\, :\, p \in C\} \leq C\bigl(|\mc P|^{\frac{k}{2k-1}+\varepsilon}|\mc C|^{\frac{2k-2}{2k-1}}+|\mc P|+|\mc C|\bigr).
\]
\end{theorem}
Suppose $\mc C$ consists of degree $\leq d$ curves with distinct irreducible components. Then any two curves intersect in at most $d^2$ points by Bezout's identity. As a consequence, if $\mc P' \subset \mc P$ has size $d^2+1$, then at most one curve from $\mc C$ contains $\mc P'$. Thus for any $\mc P$, the hypotheses are satisfied with $k = d^2+1$ and $s = d^2$. 

Let $\psi: \widehat{\C} \to \widehat{\C}$ be a M\"obius transformation, 
\begin{equation}\label{eq:MobiusTransform}
\psi(x) = \frac{ax+b}{cx+d},
\end{equation}
and consider the graph 
\[
\Gamma_{\psi} := \{(x, \psi(x))\, :\, x \in \C \setminus \{\psi^{-1}(\infty)\}\}. 
\]
This graph agrees with the finite part of the following degree two curve,
\[
\Gamma_{\psi} = \{(x, y)\, :\, y(cx+d) = ax+b\}. 
\]
Indeed, if $(x,y)$ satisfies the equation on the right hand side then $cx+d \neq 0$, as it is impossible for $ax+b=cx+d=0$, so $y = \frac{ax+b}{cx+d} = \psi(x)$.

If $c \neq 0$ then $\Gamma_{\psi}$ is (the finite part of) an irreducible conic, as it cannot contain any line. If $c = 0$, then $\Gamma_{\psi}$ is a line. Thus if $\Psi$ is a family of M\"obius transformations, $\{\gamma_{\psi}\}_{\psi \in \Psi}$ is a family of curves satisfying the hypotheses of \Cref{thm:IncidenceEst} with $k = 5$ and $s = 4$, leading to the following bound. 
\begin{corollary}\label{cor:MobiusCor}
If $\mc P \subset \C^2$ is a finite set of points and $\Psi$ is a finite set of M\"obius transformations, then 
\[
\#\{(x,y, \psi)\in \mc P \times \Psi\, :\, y = \psi(x)\}  \leq C_{\varepsilon} (|\mc P|^{\frac{5}{9}+\varepsilon}|\Psi|^{\frac{8}{9}}+|\mc P|+|\Psi|). 
\]
\end{corollary}
We will just use that 
\[
\#\{(x,y, \psi)\in \mc P \times \Psi\, :\, y = \psi(x)\}  \leq C( |\mc P|^{3/2-1/20} + |\Psi|^{3/2-1/20}). 
\]

\begin{lemma}\label{lem:expanding_lem}
Let $\psi$ be a M\"obius transformation, and let $A \subset \C^{\times} \setminus \psi^{-1}(\{0, \infty\})$ be a set of $n$ points. For some universal constant $C$,
\begin{align}
	\text{If $\psi(\{0,\infty\}) \neq \{0,\infty\}$},&\quad \max(|A : A|, |\psi(A) : \psi(A)|) \geq \frac{1}{C}n^{1+\frac{1}{40}}.  \label{lempart:expand_twopt}\\
	\text{If $\psi(\infty) \neq \infty$},&\quad \max(|A - A|, |\psi(A) - \psi(A)|) \geq \frac{1}{C}n^{1+\frac{1}{40}}.  \label{lempart:expand_onept}\\
	\text{For any $\psi$},&\quad \max(|A - A|, |\psi(A) : \psi(A)|) \geq \frac{1}{C}n^{1+\frac{1}{40}}. \label{lempart:expand_diffpt}
\end{align}
\end{lemma}
This lemma says that $A \subset \hat \C$ cannot be an approximate subgroup under two of the different group structures we could put on $\hat \C$ by removing one or two points. The special case $\psi(z) = z$ of \eqref{lempart:expand_diffpt} asserts sum-product expansion over $\C$. Using Elekes's framework for the sum product problem, we reduce \Cref{lem:expanding_lem} to \Cref{thm:IncidenceEst}. 
\begin{proof}
\textbf{Proof of \Cref{lempart:expand_twopt}}.
Consider the family of points 
\[
\mc P = (A : A) \times (\psi(A) : \psi(A))
\]
and the family of M\"obius transformations 
\[
\psi_{a,a'}(z) = \frac{\psi(az)}{\psi(a')},\qquad \Psi = \{\psi_{a,a'}\}_{a,a' \in A}.
\]
We need to show that if $(a_1, a_1') \neq (a_2, a_2')$ then $\psi_{a_1,a_1'} \neq \psi_{a_2,a_2'}$. As $\psi(\{0,\infty\}) \neq \{0,\infty\}$, either $\psi(0) \not\in \{0,\infty\}$ or $\psi(\infty) \not\in \{0,\infty\}$. Assume $\psi(0) \not\in \{0,\infty\}$ (the other case is similar). If $a_1' \neq a_2'$, then $\psi_{a_j,a_j'}$ take different values at $z = 0$. If $a_1' = a_2'$ and $a_1 \neq a_2$, then $\psi_{a_j,a_j'}$ take different values whenever $z \in \C^{\times}$. Thus $\Psi$ consists of $n^2$ distinct M\"obius transformations.

For any $(t, a,a')\in A\times A\times A$, the triple $(t/a, \psi(t)/\psi(a'), \psi_{a,a'}) = (x,y,\psi_{a,a'})$ satisfies $y = \psi_{a,a'}(x)$ (and all these triples are distinct). Thus
\[
\#\{(x,y,\psi_{a,a'}) \in \mc P \times \Psi\, :\, y = \psi_{a,a'}(x)\} \geq n^3.
\]
On the other hand, \Cref{cor:MobiusCor} implies
\[
\#\{(x,y,\psi_{a,a'}) \in \mc P \times \Psi\, :\, y = \psi_{a,a'}(x)\} \leq C(|\mc P|^{3/2-1/20} + |\Psi|^{3/2-1/20}).
\]
As $|\Psi| = n^2$, the $|\mc P|$ term dominates if $n$ is sufficiently large, giving
\[
|\mc P|\geq \frac{1}{C} n^{\frac{3}{3/2-1/20}} \geq \frac{1}{C} n^{2+\frac{1}{20}}
\]
and thus
\[
\max(|A : A|, |\psi(A) : \psi(A)|) \geq \frac{1}{C} n^{1+\frac{1}{40}}.
\]

\bigskip \noindent
\textbf{Proof of \Cref{lempart:expand_onept}.}
We let 
\[
\mc P = (A - A)  \times (\psi(A) - \psi(A))
\]
and
\[
\psi_{a,a'}(x) = \psi(x+a) - \psi(a'),\qquad \Psi = \{\psi_{a,a'}\}_{a,a'\in A}.
\]
Let $(a_1, a_1') \neq (a_2, a_2')$, and suppose that $a_1' \neq a_2'$. Because $\psi(\infty) \neq \infty$, $\psi_{a_j,a_j'}$ take different values at $z = \infty$. If $a_1' = a_2'$ and $a_1 \neq a_2$, then $\psi_{a_j,a_j'}$ take different values at all $z \in \C$. 

Given $(t, a, a') \in A\times A\times A$, the triple 
\[
(x,y,\psi_{a,a'}) = (t - a, t-a', \psi_{a,a'})
\]
satisfies $y = \psi_{a,a'}(x)$. The same argument as before implies 
\[
\max(|A - A|, |\psi(A) - \psi(A)|) \geq \frac{1}{C} n^{1+1/40}. 
\]

\bigskip \noindent
\textbf{Proof of \Cref{lempart:expand_diffpt}.}
We let 
\[
\mc P = (A - A)  \times (\psi(A) : \psi(A))
\]
and
\[
\psi_{a,a'}(x) = \frac{\psi(x+a)}{\psi(a')},\qquad \Psi = \{\psi_{a,a'}\}_{a,a'\in A}.
\]
Suppose $\psi_{a_1,a_1'} = \psi_{a_2,a_2'}$. 
As $\psi(a_j') \not\in \{0,\infty\}$, 
\[
\psi_{a_j,a_j'}^{-1}(0) = \psi^{-1}(0) - a_j,
\]
thus $a_1 = a_2$. For any $x$ such that $\psi(x+a) \not\in \{0,\infty\}$, $\psi(a_j') = \frac{\psi(x+a)}{\psi_{a_j,a_j'}(x)}$. Thus $(a_1,a_1') = (a_2,a_2')$. It follows that the family $\psi_{a,a'}$ is distinct as $(a,a')$ vary.

Given $(t,a,a') \in A\times A\times A$, the triple 
\[
(x,y,\psi_{a,a'})= (t-a, \psi(t)/\psi(a'), \psi_{a,a'})
\]
satisfies $y = \psi_{a,a'}(x)$. The same argument as before implies 
\[
\max(|A-A|, |\psi(A) : \psi(A)|) \geq \frac{1}{C} n^{1+1/40}. 
\]
\end{proof}

\begin{lemma}[Intersection of grids]\label{lem:IntersectionOfGrids}
Let $A_1$ be a $(C_{\grid}, V_1)$-grid and $A_2$ be a $(C_{\grid}, V_2)$-grid with $|A_j| \leq n$ and $V_1$, $V_2$ distinct. Then 
\[
|A_1\cap A_2| \leq C (C_{\grid} n)^{1-\frac{1}{40}}. 
\]
\end{lemma}
\begin{proof}
The closures of $V_1$ and $V_2$ are irreducible degree $\leq 3$ curves, so if $V_1$ and $V_2$ have distinct closures, $|A_1\cap A_2| \leq 9$. Assume $V_1$ and $V_2$ have equal closure. 

As $V_1$ and $V_2$ are distinct they must have common closure equal to a line or a conic, which we may identify biholomorphically with $\widehat{\C}$. Let 
\[
V_1 = \widehat{\C} \setminus s_1,\qquad V_2 = \widehat{\C} \setminus s_2
\]
where $s_1,s_2$ are distinct sets of $1$ or $2$ points. 

\bigskip \noindent 
\textbf{Case 1: $|s_1| = |s_2| = 2$.}
Without loss of generality, let $s_1 = \{0,\infty\}$, and let $\psi: \widehat{\C}\to \widehat{\C}$ be a M\"obius transformation takign $s_2 \to \{0,\infty\}$. The hypothesis then says $A_1 \subset \C^{\times}$ and $A_2 \subset \widehat{\C} \setminus s_2$ and 
\[
|A_1 : A_1| \leq C_{\grid} n,\qquad |\psi(A_2) : \psi(A_2)| \leq C_{\grid} n. 
\]
Let $B = A_1\cap A_2 \subset \C^{\times} \setminus s_2$. By Case 1 of \Cref{lem:expanding_lem}, 
\[
\max\{|B : B|, |\psi(B) : \psi(B)|\} \geq \frac{1}{C}|B|^{1+1/40}.
\]
The left hand side is $\leq C_{\grid} n$, so 
\[
|B| \leq C(C_{\grid} n)^{\frac{1}{1+1/40}} \leq C(C_{\grid} n)^{1-1/40}
\]
as desired. 

\bigskip \noindent 
\textbf{Case 2: $|s_1| = |s_2| = 1$.}
Without loss of generality, let $s_1 = \{\infty\}$ and $s_2 = \{p\}$, and let $\psi$ be a M\"obius transformation taking $p \to \infty$. Repeat the argument above and use Case 2 of \Cref{lem:expanding_lem}. 

\bigskip \noindent 
\textbf{Case 3: $|s_1| = 1$ and $|s_2| = 2$.}
Without loss of generality, let $s_1 = \{\infty\}$ and $s_2 = \{p, q\}$, and let $\psi$ be a M\"obius transformation taking $\{p,q\} \to \{0,\infty\}$. Repeat the argument above and use Case 3 of \Cref{lem:expanding_lem}.
\end{proof}

\subsection{Intermediate structure lemma}
Before proving the intermediate structure lemma, we need two lemmas. The first allows us to handle concurrent lines, and the second locate grid structure on cubics. 
\begin{lemma}\label{lem:ConclusionConcurrent}
Let $m \geq 3$. 
Suppose $A \subset \CP^2$ is a set of $n$ non-collinear points, all but $\frac{n}{m+1} - m$ of which lie on a family $\ell_1, \ldots, \ell_m$ of concurrent lines. Then $A$ admits an ordinary line.
\end{lemma}
\begin{proof}
Let $v$ be the number of points of $A$ off the family. Decrease the size of $A$ by one if necessary to assume the point of concurrency is not included. Add up to $v$ more lines to make $A$ lie on a family of concurrent lines. Assume without loss of generality that $|A \cap \ell_1|\geq \frac{n-v-1}{m}$. By \Cref{thm:SG_concurrent_lines}, if $\frac{n-v}{m} > m+v-2$, then $A$ admits an ordinary line, and we assume $v$ is small enough that this holds. 
\end{proof}

\begin{lemma}\label{lem:LemmaStructuredSet}
Let $X$ be a cubic curve, and write
\[
X = V_1 \cup V_2 \cup V_3,\qquad V_j\text{ is an irreducible component.}
\]
 Let $A_j \subset V_j \setminus \sing(X)$ be disjoint subsets with $|A_j| \leq n$ and with at least $\kappa n^2$ collinear triples in $A_1\times A_2\times A_3$. Then there exist subsets $A_j'\subset A_j$ such that $|A_j'| \geq \kappa^{C_{\BSG}+1} n$ and $A_j'$ is a $(\kappa^{C_{\BSG}}, V_j \setminus \sing(X))$-grid. 
\end{lemma}
\begin{proof}
First, each component of $X$ must appear with multiplicity equal to its degree. If $X$ is an irreducible cubic or splits into three lines, this is no additional restriction. If $X$ is splits into a conic and a line, and if the line appears with multiplicity $2$, there would not be any collinear triples, as $A_j$ avoids the intersection points. 

There are at most $|A_1| n$ collinear triples, as each pair in $A_1\times A_2$ is incident to at most one collinear triple. Thus $|A_j| \geq \kappa n$ for $j = 1,2,3$. 

By \Cref{thm:GpStructureCubic}, there is a group $G$ and biholomorphic maps $\rho_j: V_j\setminus \sing(X) \to G$ such that $(x,y,z)\in A_1 \times A_2\times A_3$ is collinear if and only if $\rho_1(x)+\rho_2(y)+\rho_3(z) = 0$. Applying \Cref{thm:BSG} to the three sets $\rho_j(A_j)$ gives the result.
\end{proof}

\begin{lemma}\label{lem:IntermediateStructureLemma}
For any $d \geq 2$ and $\delta > 0$, there exists $\varepsilon(d, \delta), n_0(d, \delta) > 0$ such that the following holds.

Let $A$ be a Sylvester-Gallai configuration of $n \geq n_0$ many points, all but $\varepsilon n$ of which lie on the smooth locus of either a union of $\leq d$ lines, or a union of $\leq (d-1)$ lines and one conic, or an irreducible cubic. 
Then in fact, all but $\delta n$ points of $A$ lie on a plane cubic. Moreover, that plane cubic is not 3 concurrent lines. 
\end{lemma}
\begin{proof}
We will need $\varepsilon$ to be small enough in terms of $\delta$ and $d$, and $n_0$ to be large enough in terms of $\delta, d, \varepsilon$. We specify the necessary conditions inside the proof. 

If all but $\varepsilon n$ points of $A$ lie on an irreducible cubic we are done, so we may assume we are in the all lines or lines and one conic case. 

Let 
\[
T := \{(x,y,z)\in A\times A\times A\, :\, \text{$x,y,z$ are distinct and collinear}\} 
\]
be the proper collinear triples in $A \times A\times A$.

Let $V = V_1 \cup \dots \cup V_m$ be the algebraic curve in the hypothesis, with $V_j$ the irreducible components, and set
\[
A_j := A\cap V_j \cap \smooth(V),\qquad A_{\err} := A \setminus \smooth(V). 
\]
Say $V_j$ is \textit{significant} if 
\[
|A_j| \geq \frac{\delta}{100d} n,
\]
and order the $V_j$ so that $V_1, \ldots, V_s$ are the significant ones, $s \leq m$. 

If the significant set is a family of concurrent lines, then as long as $\varepsilon < \frac{1}{2(d+1)}$ and $n_0$ is sufficiently large in terms of $d$, \Cref{lem:ConclusionConcurrent} supplies an ordinary line. There are two cases left to consider:
\begin{itemize}
    \item The significant set contains a conic. In this case, we show the significant set is equal to a conic and a line. 
    \item The significant set contains no conic and contains 3 non-concurrent lines. In this case, we show the significant set is equal to these 3 non-concurrent lines. 
\end{itemize}
It follows that all but $\delta n$ points of $A$ lie on a plane cubic. 

\bigskip
\noindent \textbf{Getting ready to analyze collinear triples.}
Let $V_i, V_j$ be two (possibly identitical) irreducible components of $V$, and let $Q \subset A_i$, $R \subset A_j$ be disjoint subsets. For each $p \in A_{\err}$, there are $\leq 2\min\{|Q|,|R|\}$ collinear triples in $Q\times R\times \{p\}$, as each line joining $p$ to a point of $R$ has $\leq 2$ points of $Q$ on it, and vice versa. Thus the total number of collinear triples in $Q\times R\times A_{\err}$ is at most $2 \min\{|Q|,|R|\} \varepsilon n$. If $|Q|, |R| > 4\varepsilon n$, then $2 \min\{|Q|,|R|\} \varepsilon n < \frac{1}{2}|Q||R|$. Thus if $|Q|,|R| > 4\varepsilon n$, by the pigeonhole principle there is some $k \in \{1,\ldots, m\}$ such that
\[
|T\cap (Q\times R\times A_k)| \geq \frac{1}{2d}|Q||R|.
\]
If $V_i \cup V_j \cup V_k$ is not a cubic or line, then by \Cref{thm:raz_sharir_dezeeuw}, 
\[
|T\cap (Q\times R\times A_k)| \leq C n^{11/6}. 
\]
If we choose $n_0 \geq (Cd\varepsilon^{-2})^{6}$, the same hypothesis that $|Q|, |R| \geq 4\varepsilon n$ implies there are $> Cn^{11/6}$ collinear triples. Thus $V_i \cup V_j \cup V_k$ must be a cubic or a line. If one of these is a conic, it must appear with multiplicity two in order for there to be any collinear triples.

\bigskip \noindent 
\textbf{Case 1: There is a significant conic $C$.} 
We would like to show there is at most one significant line, so suppose by way of contradiction that $\ell_1, \ell_2$ are both significant lines. The two sets $C\cap \ell_1$ and $C\cap \ell_2$ must be distinct.

\[
|T\cap (A_C\times A_C\times A_{\ell_1})| \geq   \kappa n^2,\qquad \kappa = \frac{1}{2d} \bigl(\frac{\delta}{100d}\bigr)^2. 
\]
By \Cref{lem:LemmaStructuredSet}, there is a subset $A_C' \subset A_C$ which has size $\geq \kappa^{C_{\BSG}+1} n$ and is a $(\kappa^{-C_{\BSG}}, C \setminus \ell_1)$-grid. If we choose $\varepsilon$ small enough in terms of $\delta$ and $d$, $|A_C'| \geq 4\varepsilon n$, so letting $Q = A_C'$ and $R = A_{\ell_2}$ above, we find 
\[
|T\cap (A_C'\times A_C'\times A_{\ell_2})| \geq \kappa^{C_{\BSG}+2} n.
\]
Apply \Cref{lem:LemmaStructuredSet} again to produce a subset $A_C'' \subset A_C'$ with 
\[
|A_C''| \geq \kappa^{C_{\BSG}^2+2C_{\BSG}+1} n \geq \kappa^{4C_{\BSG}^2}n
\]
and that is a $(\kappa^{-4C_{\BSG}^2}, C \setminus \ell_2)$-grid. 
By \eqref{eq:SubsetIsGrid}, $A_C''$ is also a $(\kappa^{-5C_{\BSG}^2}, C \setminus \ell_1)$-grid. 

Now apply \Cref{lem:IntersectionOfGrids} to two copies of $A_C''$, first viewed as a $C \setminus \ell_1$ grid, then viewed as a $C \setminus \ell_2$ grid. We find 
\[
|A_C''| \leq C \kappa^{-5C_{\BSG}^2} n^{1-1/40}. 
\]
If we choose $n_0$ large enough in terms of $\delta$ and $d$, this is a contradiction. 

\bigskip \noindent 
\textbf{Case 2: There is no significant conic, and there are three significant lines $\ell_1,\ell_2,\ell_3$ intersecting non-concurrently.} 

Suppose by way of contradiction there is a fourth significant line $\ell_4$. We label $\ell_i\cap \ell_j = p_{ij}$. By permuting the lines if necessary, we may assume $p_{12},p_{13},p_{14}$ are all distinct.

We make a sequence of subsets of $A_{\ell_1} \supset A_{\ell_1}^{(2)} \supset A_{\ell_1}^{(3)} \supset A_{\ell_1}^{(4)}$ as follows. First, set $Q = A_{\ell_1}$ and $R = A_{\ell_2}$ above, and find that there is some other line $\ell^{(2)}$ (possible equal to one of $\ell_1,\ldots,\ell_4$) such that 
\[
|T\cap (A_{\ell_1}\times A_{\ell_2}\times A_{\ell^{(2)}})| \geq \kappa n^2,\qquad \kappa = \frac{1}{2d} \bigl(\frac{\delta}{100d}\bigr)^2. 
\]
Apply \Cref{lem:LemmaStructuredSet} to find a subset $A_{\ell_1}^{(2)} \subset A_{\ell_1}$ such that 
\[
|A_{\ell_1}^{(2)}| \geq \kappa^{C_{\BSG}+1} n \qquad \text{and}\qquad 
A_{\ell_1}^{(2)}\text{ is a } (\kappa^{-C_{\BSG}}, \ell_1 \setminus (\ell_2 \cup \ell^{(2)}))\text{-grid}. 
\]
Assuming $\varepsilon < \frac{1}{4} \kappa^{C_{\BSG}+1}$, we may apply the same argument with $Q = A_{\ell_1}^{(2)}$ and $R = A_{\ell_3}$ to produce $A_{\ell_1}^{(3)} \subset A_{\ell_1}^{(2)}$ such that 
\[
|A_{\ell_1}^{(3)}| \geq \kappa^{4C_{\BSG}^2} n \qquad \text{and}\qquad 
A_{\ell_1}^{(3)}\text{ is a } (\kappa^{-4C_{\BSG}^2}, \ell_1 \setminus (\ell_2 \cup \ell^{(3)}))\text{-grid}. 
\]
Assuming $\varepsilon < \frac{1}{4} \kappa^{4C_{\BSG}^2}$, apply the argument once again to $Q = A_{\ell_1}^{(3)}$ and $R = A_{\ell_4}$ to produce $A_{\ell_1}^{(4)}\subset A_{\ell_1}^{(3)}$ such that 
\[
|A_{\ell_1}^{(4)}| \geq \kappa^{6C_{\BSG}^3} n \qquad \text{and}\qquad 
A_{\ell_1}^{(4)}\text{ is a } (\kappa^{-6C_{\BSG}^3}, \ell_1 \setminus (\ell_3 \cup \ell^{(3)}))\text{-grid}. 
\]

There are three sets $s_2,s_3,s_3\subset \ell_1$ of one or two points with $p_{1j} \in s_j$ such that $A_{\ell}^{(k)}$ is a $(C\kappa^{-6C_{\BSG}^3}, \ell \setminus s_j)$-grid. By \eqref{eq:SubsetIsGrid}
\[
A_{\ell}^{(4)}\text{ is a } (C\kappa^{-12C_{\BSG}^3}, \ell \setminus s_j)\text{ grid for $j = \{2,3,4\}$.}
\]
The three sets $s_2,s_3,s_4$ cannot all agree because $p_{12},p_{13},p_{14}$ are all distinct. Let $s_i,s_j$ be distinct, and apply \Cref{lem:IntersectionOfGrids} to two copies of $A_{\ell}^{(4)}$, the first viewed as a $\ell \setminus s_i$ grid, the second viewed as a $\ell \setminus s_j$ grid. We learn
\[
|A_{\ell}^{(4)}| \leq C \kappa^{-12 C_{\BSG}^3} n^{1-1/40}, 
\]
and for $n$ large enough in terms of $\delta$ and $d$ this is a contradiction.
\end{proof}

\section{Detailed structure lemma}

\begin{lemma}[Detailed structure lemma]\label{lem:DetailedStructure}
For any $d \geq 2$ and $\delta > 0$, there exists $\varepsilon(d, \delta), n_0(d, \delta) > 0$ such that the following holds. 

Let $A$ be a Sylvester--Gallai configuration of $n \geq n_0$ many points, all but $\varepsilon n$ of which lie on the smooth locus of a plane cubic $X$. Let $G$ be the group appearing in the uniformization of the smooth part of $X$. There exists a finite subgroup $H \subset G$ and an analytic map
\[
\rho: X \setminus \sing(X) \to G
\]
as in \Cref{thm:GpStructureCubic} such that 
\[
|A\Delta \rho^{-1}(H)| \leq \delta n. 
\]
\end{lemma}
\begin{proof}
Write $X$ as a union of irreducible components where each component appears with multiplicity equal to its degree,
\[
X = V_1 \cup V_2 \cup V_3.
\]
Let 
\[
A_j := A \cap V_j \cap \smooth(X),\qquad A_{\err} := A \setminus \smooth(X). 
\]
Let $\rho: X \setminus \sing(X) \to G$ be a map as in \Cref{thm:GpStructureCubic}, and let $\widetilde A_j := \rho(A_j) \subset G$. 

The sets $A_j$ must all have comparable size:
\begin{itemize}
    \item If $X$ is an irreducible cubic, the $A_j$'s are all equal.
    \item If $X = C \cup \ell$ is a conic and a line, consider an arbitrary $p \in A \setminus \ell$, and let $A_{\ell}, A_C$ be the two sets above. There are $|A_{\ell}|$ many lines joining $p$ to $A_{\ell}$, and at most $\varepsilon n$ of these meet $A_{\err}$. The rest meet $A_C$, giving $|A_C| \geq |A_{\ell}| - \varepsilon n$. Given that $|A_{\ell}|+|A_C| \geq (1-\varepsilon) n$, we learn $|A_C| \geq \frac{1}{2}n - \varepsilon n$. Let $p' \in A_C$ be arbitrary, and consider the $|A_C|-1$ lines joining $p'$ to a second point of $A_C$. As these lines are all distinct, at most $\varepsilon n$ of them may intersect $A_{\err}$, and the rest meet $A_{\ell}$ giving $|A_{\ell}| \geq |A_C| - 1 - \varepsilon n$. Thus 
    \[
    |A_C|, |A_{\ell}| \in [\frac{n}{2} - \varepsilon n - 1, \frac{n}{2}+\varepsilon n + 1]. 
    \]
    \item If $X = \ell_1\cup \ell_2\cup \ell_3$ is three non-concurrent lines, it is easy to see that each $A_j$ is non-empty. Consider $p \in A_1$. Of the $|A_2|$ many lines joining $p$ to $A_2$, at most $\varepsilon n$ may meet $A_{\err}$, so $|A_3| \geq |A_2| - \varepsilon n$. By repeating this argument it follows that $|A_j| \in [\frac{n}{3} - \varepsilon n, \frac{n}{3}+\varepsilon n]$.
\end{itemize}
The next step is to locate the subgroup $H$.
Let $p_1 \in A_1$. There are at most $\varepsilon n$ lines through $p_1$ and a point of $A_{\err}$, and each of these intersects $A_2$ in at most three points. If $p_2 \in A_2$ and the line $\ell_{p_1,p_2}$ through $p_2$ does not intersect $A_{\err}$, it must intersect $A_3$ by degree considerations. Thus by \Cref{thm:GpStructureCubic}, 
\[
\#\{(x, y) \in \widetilde A_1 \times \widetilde A_2\, :\, -x-y \not\in \widetilde A_3\} \leq |A_1| (\varepsilon n) \leq \varepsilon n^2. 
\]
If $\varepsilon$ is small enough, we may apply \Cref{prop:close_to_subgroup_prop} to the three sets $\widetilde A_1, \widetilde A_2, -\widetilde A_3 \subset G$ with $K = \varepsilon n + 1$, and we find that there is a subgroup $H \subset G$ and three cosets $x_1 + H, x_2 + H, (-x_1-x_2) + H$ such that 
\[
|\widetilde A_1 \Delta (x_1+H)|, |\widetilde A_2 \Delta (x_2+H)|, |\widetilde A_3 \Delta (-x_1-x_2+H)| \leq 7\varepsilon n + 7. 
\]
If $\varepsilon$ is small enough and if $V_j = V_{j'}$, then the corresponding cosets are equal as well. By composing $\rho$ with another biholomorphism which preserves the collinearity structure, we may assume $x_1 = x_2 = 0$, and thus $|\widetilde A_j \Delta H| \leq 7\varepsilon n + 7$. This in turn implies 
\[
|A \Delta \rho^{-1}(H)| \leq 3(7\varepsilon + 7) + \varepsilon n \leq 15\varepsilon + 7
\]
as desired. 
\end{proof}

\newpage
\section{Concluding the argument}
Combining the structural results \Cref{lem:CoarseStructureLemma}, \Cref{lem:IntermediateStructureLemma}, and \Cref{lem:DetailedStructure}, we learn that for every $d \geq 2$ and $\delta > 0$, there exists $\varepsilon(d, \delta), n_0(d, \delta) > 0$ such that if $A$ is a Sylvester--Gallai configuration with all but $\varepsilon n$ points on an algebraic curve, then there is a plane cubic $X$ (which is not 3 concurrent lines), a map $\rho: X \setminus \sing(X) \to G$ as in \Cref{thm:GpStructureCubic}, and a finite subgroup $H \subset G$ such that 
\begin{equation}\label{eq:ACloseToH}
|A \Delta \rho^{-1}(H)| \leq \delta n. 
\end{equation}

We can immediately say that $A$ does not intersect the singular locus of $X$. If $p \in A \cap \sing(X)$, then any line through $p$ is either a component of $X$ or only intersects $X$ in one other point. There are $\geq \frac{n}{3}-2\delta n$ lines through $p$ and a point of $A\cap \smooth(X)$ which are not a component, and at most $\varepsilon n$ of these can meet a third point of $A$, so if $\delta$ is small enough there would have to be an ordinary line.

If $G = \C$, then there are no finite subgroups, so the conclusion is impossible. This rules out the cases where $X$ is a cuspidal cubic or a conic and a line intersecting in one point. The remaining cases are:
\begin{itemize}
    \item $X$ is a conic and a line intersecting in two points, 
    \item $X$ is a nodal cubic, 
    \item $X$ is a smooth irreducible cubic,
    \item $X$ is 3 non-concurrent lines.
\end{itemize}
In the first three cases we show $X$ admits an ordinary line, so cannot be a Sylvester--Gallai configuration. In the last case we show $X$ is projectively equivalent to a Fermat configuration. First we need a lemma.

\begin{lemma}\label{lem:TangentsCount}
Let $X$ be a plane cubic and let $p \in \CP^2 \setminus X$. At most $9$ lines through $p$ intersect $X$ in fewer than $3$ points. 
\end{lemma}
\begin{proof}
Let $f(x,y,z)$ be the homogeneous degree 3 defining polynomial of $X$. Let $p = [x_0 : y_0 : z_0]$, and let $q \in X$. The line through $p$ and $q$ intersects $X$ with multiplicity $1$ at $q$ if the following cross product is nonzero, 
\[
(\partial_x f, \partial_y f, \partial_z f) \times (x - x_0, y - y_0, z - z_0) \neq 0. 
\]
At least one of the three coordinates in this cross product defines a degree $3$ curve that has no overlapping connected component with $X$. Thus by Bezout's inequality, it intersects $X$ in at most $9$ points.

If a line through $p$ intersects $X$ in fewer than $3$ points, by Bezout's theorem it would have to intersect one of them with multiplicity at least $2$. Thus there are at most $9$ lines through $p$ intersecting $X$ in fewer than $3$ points. 
\end{proof}
One could obtain a more precise count using the Riemann--Hurwitz formula, but we do not need this. 

\subsection{$X$ is the union of a conic and a line intersecting in two points}
Let $X = C \cup \ell$.
Define $H' \subset H$ to be the set of $x \in H$ such that both $\rho|_C^{-1}(x) \in A$ and $\rho|_{\ell}^{-1}(-2x) \in A$. This set has size at least
\[
|H'| \geq |H| - 2\delta n \geq n - 3\delta n.
\]
Consider the family of lines
\[
L = \left\{ \ell_{\rho|_C^{-1}(x),\, \rho|_{\ell}^{-1}(-2x)} : x \in H' \right\}.
\]
Each line in $L$ is tangent to $C$, and each intersects $A \cap X$ in exactly two points. By \Cref{lem:TangentsCount}, each $p \in A \setminus C$ meets at most $9$ of these lines, so at most $9\varepsilon n$ lines of $L$ can meet a third point of $A$. If $\varepsilon$ is sufficiently small, some line of $L$ passes through exactly two points.

\subsection{$X$ is a nodal or smooth cubic}
Let $\rho: X \setminus \sing(X) \to G$ be the map described above. Let $H' \subset H$ be the set of $x \in H$ such that both $\rho^{-1}(x) \in A$ and $\rho^{-1}(-2x) \in A$, and $3x \neq 0$ (so that $x$ and $-2x$ are distinct). This set has size at least 
\[
|H'| \geq |H| - 9 -  2\delta n \geq n - 9 - 3 \delta n. 
\]
Consider the family of lines
\[
L = \left\{ \ell_{\rho^{-1}(x),\, \rho^{-1}(-2x)} : x \in H' \right\}.
\]
Each line in $L$ is tangent to a point of $X$, and each intersects $A \cap X$ in exactly two points. By the same argument as above, if $\varepsilon$ is sufficiently small, some line of $\ell$ is an ordinary line.

\subsection{$X$ is the union of 3 non-concurrent lines}
We first show $A$ is contained in the smooth part of $\ell_1\cup \ell_2\cup \ell_3$, and then show $A$ is projectively equivalent to a Fermat configuration.

\bigskip \noindent 
\textbf{Part 1: $A$ is contained in the smooth part of $\ell_1\cup \ell_2\cup \ell_3$.}

Let 
\[
A_j := A \cap \rho^{-1}(H) \cap \ell_j,\qquad A_{\err} := A \setminus (A_1\cup A_2\cup A_3). 
\]
Each set $A_j$ has $|A_j| \in [n/3 - \delta n, n/3+\delta n]$. We also let 
\[
p_{ij} := \ell_i\cap \ell_j. 
\]
Suppose $p \in A \setminus (\ell_1\cup \ell_2\cup \ell_3)$.
Consider the lines joining $p$ to $A_1$. At most $\delta n$ of these meet $A_{\err}$, so there is a subset $A_1' \subset A_1$ with size at least $\frac{1}{2}|A_1| - \delta n$ such that either every line through $p$ and $A_1'$ meets $A_2$ or every line through $p$ and $A_1'$ meets $A_3$. Suppose without loss of generality they meet $A_2$. 
Let $\varphi_p: \ell_1\to \ell_2$ be the map
\[
\varphi_p(q) = \ell_{p,q}\cap \ell_2. 
\]
Due to the ambient group $H$, $A_1'$ is a $(\frac{|H|}{|A_1|'}, \ell_1 \setminus \{\ell_{12}, \ell_{13}\})$-grid. If $\delta$ is small enough, $\frac{|H|}{|A_1|'} \leq 3$. As $\varphi_p$ is a biholomorphism, $\varphi_p(A_1')$ is a $(3, \ell_2 \setminus \varphi_p(\{\ell_{12}, \ell_{13}\}))$-grid. But $\varphi_p(A_1') \subset A_2$, so $A_2$ is a $(3, \ell_2 \setminus \{\ell_{23}, \ell_{12}\})$-grid. By \Cref{lem:IntersectionOfGrids}, 
\[
|A_1'| \leq C n^{1-1/40},
\]
leading to a contradiction. 

\bigskip \noindent 
\textbf{Part 2: $A$ is projectively equivalent to a Fermat configuration.}
Kelly and Nwankpa proved over general fields \cite{KellyNwankpa}*{Theorem 3.11} that a Sylvester--Gallai configuration lying on three non-concurrent lines is projectively equivalent to a Fermat configuration. We include a proof in our case for completeness.

The first step is to show $A = \rho^{-1}(H)$. We could use the $\varepsilon = 0$ case of \Cref{prop:close_to_subgroup_prop}, but we include a proof because it is easy. 

Let $\widetilde A_j = \rho(A\cap \ell_j) \subset \C^{\times}$. (We switch to using product notation because we are working in $\C^{\times}$). 
The fact that $A$ is a Sylvester--Gallai configuration implies 
\[
\widetilde A_1\widetilde A_2 \subset \frac{1}{\widetilde A_3},
\]
and similarly if the order is permuted, so $|\tilde A_1| = |\tilde A_2| = |\tilde A_3|$. Let $x,x',x''\in \widetilde A_1$, and let $y_0 \in \widetilde A_2$. Using the product set property several times shows
\[
(xy_0)^{-1} = \frac{1}{xy_0} \in \widetilde A_3 \Rightarrow \bigl(x'\frac{1}{xy_0}\bigr)^{-1}  = \frac{xy_0}{x'}\in \widetilde A_2 \Rightarrow \bigl(x''\frac{xy_0}{x'}\bigr)^{-1} = \frac{x'}{x''xy_0} \in \widetilde A_3 \Rightarrow \bigl(y_0\frac{x'}{x''xy_0}\bigr)^{-1} = \frac{x''x}{x'} \in \widetilde A_1. 
\]
This implies $A_j = \lambda_jH_j$ for some finite subgroup $H_j\subset \C$ and elements $\lambda_j \in \C^{\times}$. The product set property implies the $A_j$ are cosets of a common subgroup $H_m$, 
\[
A_j = \lambda_j H_m,\qquad \lambda_j \in \C, H_m \text{ is the group of $m$th roots of unity.}
\]
The product set property implies $\lambda_1\lambda_2\lambda_3 = 1$.

The space of projective transformations of $\CP^2$ is transitive on sets of 3 non-concurrent lines. Apply a projective transformation so that 
\[
\ell_1 = \{[0 : y : z]\},\quad \ell_2 = \{[x : 0: z]\},\quad \ell_3 = \{[x : y : 0]\}. 
\]
The three maps
\[
\rho_1([0 : y : z]) = y/z,\quad \rho_2([x : 0 : z]) = -z / x,\quad \rho_3([x : y : 0]) = x/y,
\]
satisfy that $(p,q,r) \in (\ell_1\setminus \{p_{12}, p_{13}\})\times (\ell_2\setminus \{p_{12}, p_{23}\})\times (\ell_3\setminus \{p_{13}, p_{23}\})$ are collinear if and only if $\rho_1(p)\rho_2(q)\rho_3(r) = 1$. The projective transformation 
\[
[x : y : z] \mapsto [\lambda_2^{-1} x : \lambda_1 y :  z]
\]
has the effect of dilating $\rho_1(p)$ by $\lambda_1$, $\rho_2(q)$ by $\lambda_2$, and $\rho_3(r)$ by $\frac{1}{\lambda_1\lambda_2} = \lambda_3$. Thus after applying a projective transformation and letting $\zeta$ denote an $m$th root of unity,
\begin{align*}
    A_1 &= \rho_1^{-1}(H_m) = \{[0 : \zeta^j : 1]\}_{j=0}^{m-1}, \\ 
    A_2 &= \rho_2^{-1}(H_m) = \{[-\zeta^j : 0 : 1]\}_{j=0}^{m-1},  \\ 
    A_3 &= \rho_3^{-1}(H_m) = \{[\zeta^j : 1 : 0]\}_{j=0}^{m-1},
\end{align*}
which is the Fermat configuration on $n = 3m$ points.
This completes the proof of \Cref{thm:MainThm}.

\appendix
\section{Proof of the group law on a cubic}\label{sec:ProofOfGpLaw}
This proof of \Cref{thm:GpStructureCubic} generalizes the proof for nonsingular cubics using the Jacobian. 

Let $X^* := \smooth(X)$. 
Let $\Div(X^*)$ be the free abelian group generated by the set $X^*$. That is, 
\[
\Div(X) = \{\sum_{j=1}^n a_j P_j\, :\, P_j \in X^*\text{ and } a_j \in \Z\},
\]
with the pointwise addition. Let $\Div^0(X^*)$ be the divisors that have degree zero in each component of $X^*$, where the degree in component $V$ is $\sum_{P_j\in V} a_j$. 

If $\ell$ is a line not intersecting $\sing(X)$, then $\ell$ meets $X^*$ in a degree 3 divisor which we denote by $\Div(\ell \cap X^*)$. Let $\mc R \subset \Div(X^*)$ be the subgroup generated by $\Div(\ell \cap X^*)$ where $\ell$ ranges over all such lines, and let $\mc R^0$ be the degree zero part of $\mc R$. We remark that every divisor in $\mc R$ has multiplicity on each component equal to its degree, so having degree zero on each component is equivalent to having degree zero overall. Also, recall that any line intersecting $\sing(X)$ is either a component of $X$ or intersects at most one more point of $X$. Thus for any $P, Q \in X^*$ not both lying on a line $\ell$ that is a component of $X^*$, there is a unique divisor $P+Q+A \in \mc R$. 

Set 
\[
J = \Div^0(X^*) / \mc R^0. 
\]
Given a divisor $s \in \Div^0(X^*)$, we denote by $[s]$ its image in $J$. The following two lemmas allow us to identify $J$ with $G$, where $G$ is the uniformization of a component of $X^*$.

\begin{lemma}
Let $P \in X^*$ be arbitrary. Every element of $J$ has a representative of the form $[P - Q]$ where $Q$ is in the same component as $P$. 
\end{lemma}
\begin{proof}
Let $s = \sum a_j P_j$ be a degree zero divisor. Let $|s| = \sum |a_j|$ be the absolute degree. 

Suppose $|s| > 2$, and suppose that $a_1, a_2 < 0$ and $a_3 > 0$ (the argument is similar if $a_1,a_2 > 0$) and $P_1, P_3$ lie on the same component of $X^*$.  
We want to make sure that $P_1,P_2$ do not define a line that is a component of $X$. If they do, let $Q$ be an arbitrary, generic point on a different component, and let $P_1+Q+P_1'$ and $P_3+Q+P_3'$ be in $\mc R$. Then 
\[
[s] = [s] + [P_1+Q+P_1' - P_3 - Q - P_3'] = [s + P_1 - P_3 + P_1' - P_3'].
\]
Let $s' = s + P_1 - P_3 + P_1' - P_3'$. Then $|s'| = |s|$, $s'$ takes negative values on $P_3'$ and $P_2$, and takes a positive value on $P_1'$. $P_3'$ and $P_1'$ still lie on the same component, but $P_3'$ and $P_2$ do not. We change $s$ to $s'$ and $(P_1,P_2,P_3)$ to $(P_3',P_2,P_1')$. 

The line through $P_1$ and $P_2$ defines a divisor $P_1+P_2+A \in \mc R$. As $P_3$ lies on the same component as $P_1$, $A$ and $P_3$ do not define a line component of $X$. The line through $A$ and $P_3$ defines a divisor $A+P_3+B \in \mc R$. Thus $P_1+P_2-P_3-B \in \mc R$, and $[s] = [s + P_1+P_2-P_3-B]$. We have 
\[
|s+P_1+P_2-P_3-B| \leq |s| - 2,
\]
allowing us to decrement the absolute degree. 

Continuing in this way, we may assume $|s| = 2$. Because the degree of $s$ is zero in each component, $s = P' - Q'$ for some $P', Q'$ in the same component of $X^*$. If $P' = P$ we are done. If $P' \neq P$ but lies in the same component, use the same trick as before to move $P', Q'$ to a different component while staying in the same equivalence class. Now let $P, Q'$ define a divisor $P+Q'+A$, and let $P'$ and $A$ define a divisor $P'+A+B$. Then 
\[
[s] = [P'-Q' + P+Q'+A-P'-A-B] = [P - B].
\]
as desired.
\end{proof}

\begin{lemma}
For any distinct $P, Q$ in the same component of $X^*$, $[P - Q] \neq 0$.
\end{lemma}
\begin{proof}
If $[P - Q] = 0$, then there exists a collection of lines $\ell_1, \ldots, \ell_{2m}$ none of which intersect $\sing(X)$ such that 
\[
\sum_{j=1}^m \Div(\ell_j \cap X^*) - \sum_{j=m+1}^{2m} \Div(\ell_j\cap X^*) = P - Q. 
\]
Let $f_j: \CP^2 \to \C$ be a defining equation for the line $\ell_j$, and let 
\[
f = \frac{f_1\dots f_m}{f_{m+1}\dots f_{2m}} : \CP^2 \to \widehat{\C}. 
\]
This is a rational function which has a simple zero at $P$ and a simple pole at $Q$ and no other zeros or poles on $X^*$. 

Identify the component of $X^*$ containing $P$ and $Q$ with a copy of $G$, and let $\tilde f: G \to \widehat{\C}$ be the restriction of $f$ in these coordinates. This is a meromorphic function with one zero and one pole. We split into cases depending on $G$.
\begin{itemize}
    \item If $G = \C$, then $\tilde f$ must take the same value with multiplicity two at $\infty$. First, if $X$ is 3 concurrent lines, then $f$ has no zeros or poles on the two lines not containing $P$ and $Q$, so $f$ must take a common constant value on there. Second, if $X$ is a conic and a line meeting in one point, then it must be constant on the component not containing $P$ and $Q$. Third, if $X$ is a cuspidal cubic, then this follows from $f$ being holomorphic in a neighborhood of the singular point. 

    We know $\tilde f = c \frac{z-\alpha}{z-\beta}$ for some distinct $\alpha, \beta \in \C$, so the equation 
    \[
    c\frac{z-\alpha}{z-\beta} = r
    \]
    is degree one in $z$, so cannot have derivative zero at infinity. 

    \item If $G = \C^{\times}$, then $\tilde f$ must take the same value at $0$ and $\infty$, for similar reasons to the prior case. This is once again incompatible with $\tilde f$ having degree one. 

    \item If $G = \C / \Lambda$, then any meromorphic function $\tilde f$ on $G$ satisfies
    \[
    \sum_{p\in G} p\, \mathrm{ord}_p \tilde f(p) = 0
    \]
    with the sum considered in $G$. If $\tilde f$ has a zero at $P$ and a pole at $Q$ this implies $P = Q$ which is impossible. 
\end{itemize}
\end{proof}

Combining the above lemmas, we see that for any $P_0 \in X^*$, if $V^*$ is the connected component of $P_0$ in $X^*$, then the map $\beta: V^* \to J$ by $\beta(Q) := [P_0 - Q]$ is a bijection. 

\begin{lemma}\label{lem:HomomorphismLemma}
Let $P_0 \in X^*$ and let $V^*$ be the component of $P_0$. Let $\kappa: V^* \to G$ be a biholomorphism sending $P_0 \to 0$. Then $\beta \circ \kappa^{-1}: G \to J$ is a bijection and a group homomorphism. 
\end{lemma}
\begin{proof}
Let $F: G\times G \to G$ be the map sending 
\[
[2P_0 - \kappa^{-1}(x) - \kappa^{-1}(y)] = [P_0 - \kappa^{-1}(F(Q_1, Q_2))]. 
\]
We would like to show $F(x,y)=x+y$.
First we show $F$ is holomorphic. Let $\tilde F: V^* \times V^*\to V^*$ be the analagous map to $F$ defined on $V^*$. We must show $\tilde F$ is holomorphic locally. 
Let $A \in X^*$ be a generic auxiliary point on a different component from $P_0$. Let $P_0 + A + P_0' \in \mc R$, $A + Q_1 + Q_1' \in \mc R$, $P_0' + Q_2 + B \in \mc R$, $Q_1' + B + C \in \mc R$. Then 
\begin{align*}
[2P_0 - Q_1 - Q_2] &= [2P_0 - Q_1 - Q_2] - [A + P_0 + P_0']  + [A + Q_1 + Q_1'] + [P_0' + Q_2 + B] - [Q_1' + B + C] \\ 
&= [P_0 - C]. 
\end{align*}
Generically in $A$, all the collinear triples involved have all distinct points, so the map taking $(Q_1, Q_2) \to C$ is holomorphic. 

Second, it is clear $F(0, y) = y$ and $F(x, 0) = x$. 

Third, for each nonzero $x_0 \in G$, $y \mapsto F(x_0, y)$ is a holomorphic function from $G$ to $G$ with no fixed points. If there were a fixed point, then there would be some $Q_1 \neq P_0$ and some $Q_2$ such that 
\[
[2P_0 - Q_1 - Q_2] = [P_0 - Q_2] \Longrightarrow [P_0 - Q_1] = 0
\]
contradicting bijectivity of $\beta$. Thus $g(y) = F(x_0, y) - y$ is a holomorphic function from $G\to G$ which never takes the value $\mathrm{Id}\in G$. If $G = \C / \Lambda$, this implies $g$ is constant, as any non-constant holomorphic map from $G$ to $G$ is surjective. 
One may check from geometric considerations that if $G = \C$ then $\lim_{y \to \infty} F(x_0, y) = \infty$ and thus $g(y) = F(x_0, y) - y$ extends to a holomorphic function on $\widehat{\C}$, and must be constant. Similarly, if $G = \C^{\times}$ then $\lim_{y \to 0} F(x_0,y) = 0$ and $\lim_{y\to \infty} F(x_0,y) = \infty$, so $g(y) = \frac{F(x_0,y)}{y}$ extends to a holomorphic function on $\widehat{\C}$ and must be constant. Plugging in $y = 0$ gives $g(0) = F(x_0, 0) = x_0$, thus 
\[
F(x_0,y) = x_0+y
\]
as desired.
\end{proof}

Now we are ready to prove \Cref{thm:GpStructureCubic}. 

\begin{proof}[Proof of \Cref{thm:GpStructureCubic}]
Let $X^* = V_1^* \cup V_2^* \cup V_3^*$ where each component appears with multiplicity equal to its degree. 
Choose $P_j \in V_j^*$ such that 
\begin{itemize}
    \item $P_1+P_2+P_3 \in \mc R$, 
    \item each component has just one distinct point on it, but with multiplicity equal to its degree. 
\end{itemize}
If $X^*$ is irreducible we choose $P$ to be an inflection point and take $P_1 = P_2 = P_3 = P$. If $X^*$ is a conic $C$ and a line $\ell$ we let $P \in C^*$, take $Q$ to be the intersection of the tangent to $C$ through $P$ with $\ell$,  and set $P_1 = P_2 = P$, $P_3 = Q$. If $X^*$ is a union of three lines, we just take $P_1,P_2,P_3$ to be three collinear points. 

Let $P_0 \in X^*$ be arbitrary, and let $V^*$ be the component of $P_0$. Let $\kappa: V^* \to G$ be a biholomorphism sending $P_0$ to the identity. Define $\rho: V_j^* \to G$ by 
\[
[P_j - Q] = [P_0 - \kappa^{-1}(\rho(Q))].
\]
Let $(Q_1,Q_2,Q_3) \in V_1^*\times V_2^*\times V_3^*$. By \Cref{lem:HomomorphismLemma}, 
\[
[P_1 - Q_1] + [P_2 - Q_2] + [P_3 - Q_3] = [P_0 - \kappa^{-1}(\rho(Q_1)+\rho(Q_2)+\rho(Q_3))]. 
\]
If $Q_1+Q_2+Q_3 \in \mc R$, then the left hand side equals zero in $J$.  If $Q_1+Q_2+Q_3 \neq 0$, then there is some $Q_3 \neq Q_3' \in V_3^*$ such that $Q_1+Q_2+Q_3' \in \mc R$, in which case
\[
[P_1 - Q_1] + [P_2 - Q_2] + [P_3 - Q_3] = [Q_3' - Q_3] \neq 0,
\]
so the left hand side is nonzero. Thus $\rho(Q_1)+\rho(Q_2)+\rho(Q_3) = 0$ if and only if $Q_1+Q_2+Q_3 \in \mc R$. In particular, if $Q_1,Q_2,Q_3$ are distinct, they are collinear if and only if this sum vanishes. 
\end{proof}

\bibliography{references}

\end{document}